\documentclass[12pt, reqno]{amsart}

\author[M.~Caprio]{Michele Caprio, Andrea Aveni and Sayan Mukherjee}
\address{Department of Statistical Science,
Duke University, 214 Old Chemistry, Durham, NC 27708-0251}
\email{michele.caprio@duke.edu}
\urladdr{\url{https://mc6034.wixsite.com/caprio}} 
\address{Department of Statistical Science,
Duke University, 214 Old Chemistry, Durham, NC 27708-0251}
\email{andrea.aveni@duke.edu}
\urladdr{\url{https://www.researchgate.net/profile/Andrea_Aveni}} 
\address{Department of Statistical Science, Mathematics, Computer Science, and Biostatistics \& Bioinformatics, Duke University, Durham, NC 27708-0251}
\email{sayan@stat.duke.edu}
\urladdr{\url{https://sayanmuk.github.io/}}

\keywords{Non-Diophantine arithmetics; convergence of series; paradox of the heap.}
\subjclass[2010]{Primary: 03H15; Secondary: 03C62.}

\title{Concerning three classes of Non-Diophantine arithmetics}

\usepackage[T1]{fontenc}
\usepackage{amsmath}
\usepackage{setspace}
\usepackage{amssymb}
\usepackage{amsthm}
\usepackage[left=2.5cm, right=2.5cm, top=3cm]{geometry}
\usepackage{hyperref}
\usepackage{fancyhdr}
\usepackage[inline]{enumitem}
\usepackage{comment}
\usepackage{nicefrac}
\usepackage{bm}
\usepackage{mathrsfs}
\usepackage{graphicx}
\usepackage[utf8]{inputenc}
\usepackage{cancel}
\usepackage{mathtools}
\usepackage{mathabx}

\newcommand{\vertiii}[1]{{\left\vert\kern-0.25ex\left\vert\kern-0.25ex\left\vert #1 
    \right\vert\kern-0.25ex\right\vert\kern-0.25ex\right\vert}}
		
\AtBeginDocument{%
   \def\MR#1{}
}

\newtheorem{thm}{Theorem}[section]

\theoremstyle{definition} 
\newtheorem{defi}[thm]{Definition}
\let\olddefi\defi
\renewcommand{\defi}{\olddefi\normalfont}
\newtheorem{example}[thm]{Example}
\let\oldexample\example
\renewcommand{\example}{\oldexample\normalfont}
\newtheorem{rmk}[thm]{Remark}
\let\oldrmk\rmk
\renewcommand{\rmk}{\oldrmk\normalfont}

\theoremstyle{theorem}
\newtheorem{proposition}{Proposition}

\theoremstyle{definition}
\newtheorem*{definition}{Definition}
\newtheorem*{remark}{Remark}

\theoremstyle{theorem}
\newtheorem{lemma}{Lemma}

\theoremstyle{theorem}
\newtheorem{corollary}{Corollary}

\hypersetup{
    pdftitle={prove},
    pdfauthor={autori},
    pdfmenubar=false,
    pdffitwindow=true,
    pdfstartview=FitH,
    colorlinks=true,
    linkcolor=blue,
    citecolor=green,
    urlcolor=cyan
}

\providecommand{\MR}[1]{}

\providecommand{\MR}{\relax\ifhmode\unskip\space\fi MR }

\providecommand{\href}[2]{#2}

\begin{document}

\begin{abstract}
We present three classes of abstract prearithmetics, $\{\mathbf{A}_M\}_{M \geq 1}$, $\{\mathbf{A}_{-M,M}\}_{M \geq 1}$, and $\{\mathbf{B}_M\}_{M > 0}$. The first one is weakly projective with respect to the nonnegative real Diophantine arithmetic $\mathbf{R_+}=(\mathbb{R}_+,+,\times,\leq_{\mathbb{R}_+})$, the second one is weakly projective with respect to the real Diophantine arithmetic $\mathbf{R}=(\mathbb{R},+,\times,\leq_{\mathbb{R}})$, while the third one is projective with respect to the extended real Diophantine arithmetic $\overline{\mathbf{R}}=(\overline{\mathbb{R}},+,\times,\leq_{\overline{\mathbb{R}}})$. In addition, we have that every $\mathbf{A}_M$ and every $\mathbf{B}_M$ are a complete totally ordered semiring, while every $\mathbf{A}_{-M,M}$ is not. We show that the projection of any series of elements of $\mathbb{R}_+$ converges in $\mathbf{A}_M$, for any $M \geq 1$, and that the projection of any non-oscillating series series of elements of $\mathbb{R}$ converges in $\mathbf{A}_{-M,M}$, for any $M \geq 1$, and in $\mathbf{B}_M$, for all $M > 0$. We also prove that working in $\mathbf{A}_M$ and in $\mathbf{A}_{-M,M}$, for any $M \geq 1$, and in $\mathbf{B}_M$, for all $M>0$, allows to overcome a version of the paradox of the heap.
\end{abstract}

\maketitle
\thispagestyle{empty}

\section{Introduction.}
Although the conventional arithmetic---which we call Diophantine from Diophantus, the Greek mathematician who first approached this branch of mathematics---is almost as old as mathematics itself, it sometimes fails to correctly describe  natural phenomena. For example, in \cite{helmholtz} Helmoltz points out that adding one raindrop to another one leaves us with one raindrop, while in \cite{kline} Kline notices that Diophantine arithmetic fails to correctly describe  the result of combining gases or liquids by volume. Indeed, one  quarter of alcohol and one quarter of water only yield about 1.8 quarters of vodka. To overcome this issue, scholars started developing inconsistent arithmetics, that is, arithmetics for which one or more Peano axioms were at the same time true and false. The most striking one was ultraintuitionism, developed by Yesenin-Volpin in \cite{yesenin}, that asserted that only a finite quantity of natural numbers exists. Other authors suggested that numbers are finite (see e.g. \cite{rosinger} and \cite{vanbendegem}), while different scholars adopted a more moderate approach. The inconsistency of these alternative arithmetics lies in the fact that they are all grounded in the ordinary Diophantine arithmetic. The first consistent alternative to Diophantine arithmetic was proposed by Burgin \cite{Burgin1977}, and the name non-Diophantine seemed perfectly suited for this arithmetic. Non-Diophantine arithmetics for natural and whole numbers have been studied by Burgin in \cite{Burgin1977, Burgin2020, Burgin2019, Burgin2017}, while those for real and complex numbers by Czachor in \cite{Aerts,Czachor2016}. A complete account on non-Diophantine arithmetics can be found in the recent book by Burgin and Czachor \cite{Burgin2020_2}.


There are two types of non-Diophantine arithmetics: dual and projective. In this paper, we work with the latter. We start by defining an abstract prearithmetic $\mathbf{A}$,
$$\mathbf{A}:=(A,+_A,\times_A,\leq_A),$$where $A\subset \mathbb{R}$ is the \textit{carrier} of $\mathbf{A}$ (that is, the set of the elements of $\mathbf{A}$), $\leq_A$ is a partial order on $A$, and $+_A$ and $\times_A$ are two binary operations
 defined on the elements of $A$. We conventionally call them addition and multiplication, but that can be any generic operation. Naturally, the conventional Diophantine arithmetic $\mathbf{R}=(\mathbb{R},+,\times,\leq_{\mathbb{R}})$ of real numbers 
 is an abstract preartithmetic.\footnote{Notice that $+$ is the usual addition, $\times$ is the usual multiplication, and $\leq_{\mathbb{R}}$ is the  usual partial order on $\mathbb{R}$.} 

Abstract prearithmetic $\mathbf{A}$ is called \textit{weakly projective} with respect to a second abstract prearithmetic $\mathbf{B}=(B,+_B,\times_B,\leq_B)$ if there exist two functions $g:{A} \rightarrow {B}$ and $h:{B} \rightarrow {A}$ such that, for all $a,b \in A$,
\begin{align*}
a +_A b = h(g(a) +_B g(b)) \quad \text{ and } \quad
a \times_A b = h(g(a) \times_B g(b)).
\end{align*}
Function $g$ is called the \textit{projector} and function $h$ is called the \textit{coprojector} for the pair $(\mathbf{A}, \mathbf{B})$. 

The \textit{weak projection} of the sum $a+_B b$ of two elements of $B$ onto $A$ is defined as $h(a+_B b)$, while the weak projection of the product $a \times_B b$ of two elements of $B$ onto $A$ is defined as $h(a \times_B b)$.

Abstract prearithmetic $\mathbf{A}$ is called \textit{projective} with respect to abstract prearithmetic $\mathbf{B}$ if it is weakly projective with respect to $\mathbf{B}$, with projector $f^{-1}$ and coprojector $f$. We call $f$, that has to be bijective, the \textit{generator} of projector and coprojector.

Weakly projective prearithmetics depend on two functional parameters, $g$ and $h$---one, $f$, if they are projective---and recover the conventional Diophantine arithmetic when these functions are the identity.
 To this extent, we can consider non-Diophantine arithmetics as a generalization of the Diophantine one.

In this work, we consider three classes of abstract prearithmetics, $\{\mathbf{A}_M\}_{M \geq 1}$, $\{\mathbf{A}_{-M,M}\}_{M \geq 1}$, and $\{\mathbf{B}_M\}_{M > 0}$. They are useful to describe some natural and tech phenomena for which the conventional Diophantine arithmetic fails, and their elements allow us to overcome the version of the paradox of the heap (or sorites paradox) stated in \cite[Section 2]{Burgin2017}. The setting of this variant of the sorites paradox is adding one grain of sand to a heap of sand, and the question is, once a grain is added, whether the heap is still a heap. The heart of sorites paradox is the issue of vagueness, in this case vagueness of the word ``heap''.  

We show that every element $\mathbf{A}_{M}$ of the first class is a complete totally ordered semiring, and it is weakly projective with respect to $\mathbf{R_+}$. Furthermore, we prove that the weak projection of any series $\sum_n a_n$ of elements of $\mathbb{R}_+:=[0,\infty)$ is  convergent in each $\mathbf{A}_{M}$. 

The elements $\mathbf{A}_{-M,M}$ of the second class allow to overcome the paradox of the heap and are weakly projective with respect to the real Diophantine arithmetic $\mathbf{R}=(\mathbb{R},+,\times,\leq_{\mathbb{R}})$. The weak projection of any non-oscillating series $\sum_n a_n$ of terms in $\mathbb{R}$ is convergent in $\mathbf{A}_{-{M^\prime},{M^\prime}}$, for all ${M^\prime} \geq 1$. The drawback of working with this class is that its elements are not semirings, because the addition operation is not associative. 

The elements of $\{\mathbf{B}_{M}\}_{M > 0}$ too can be used to solve the paradox of the heap. They are complete totally ordered semirings, and are projective with respect to the extended real Diophantine arithmetic $\overline{\mathbf{R}}=(\overline{\mathbb{R}},+,\times,\leq_{\overline{\mathbb{R}}})$. The projection of any non-oscillating series $\sum_n a_n$ of terms in $\mathbb{R}$ is convergent in $\mathbf{B}_{{M^\prime}}$, for all ${M^\prime} > 0$. 

\begin{table}[h!]
\begin{spacing}{1.5}
\resizebox{\linewidth}{!}{
\begin{tabular}{lll}
\hline
$\mathbf{A}_M$                                                                                                                           & $\mathbf{A}_{-M,M}$                                                                                                                                   & $\mathbf{B}_M$                                                                                                                                   \\ \hline
\multicolumn{1}{|l|}{$A_M=[0,M]$}                                                                                                        & \multicolumn{1}{l|}{$A_{-M,M}=[-M,M]$}                                                                                                                & \multicolumn{1}{l|}{$B_M=[0,M]$}                                                                                                                 \\ \hline
\multicolumn{1}{|l|}{Solves paradox of the heap}                                                                                         & \multicolumn{1}{l|}{Solves paradox of the heap}                                                                                                       & \multicolumn{1}{l|}{Solves paradox of the heap}                                                                                                  \\ \hline
\multicolumn{1}{|l|}{Weakly projective wrt $\mathbf{R}_+$}                                                                               & \multicolumn{1}{l|}{\begin{tabular}[c]{@{}l@{}}Weakly projective wrt $\mathbf{R}$\end{tabular}}                                                    & \multicolumn{1}{l|}{\begin{tabular}[c]{@{}l@{}}Projective wrt $\overline{\mathbf{R}}$\end{tabular}}                                           \\ \hline
\multicolumn{1}{|l|}{Complete totally ordered semiring}                                                                                  & \multicolumn{1}{l|}{Not a semiring}                                                                                                                   & \multicolumn{1}{l|}{Complete totally ordered semiring}                                                                                           \\ \hline
\multicolumn{1}{|l|}{\begin{tabular}[c]{@{}l@{}}Weak projection of a series of elements of \\ $\mathbb{R}_+$ is convergent\end{tabular}} & \multicolumn{1}{l|}{\begin{tabular}[c]{@{}l@{}}Weak projection of a series of elements of \\ $\mathbb{R}$ is absolutely convergent\end{tabular}}      & \multicolumn{1}{l|}{\begin{tabular}[c]{@{}l@{}}Projection of a series of elements of \\ $\mathbb{R}$ is absolutely convergent\end{tabular}}      \\ \hline
\multicolumn{1}{|c|}{--}                                                                                                                 & \multicolumn{1}{l|}{\begin{tabular}[c]{@{}l@{}}Weak projection of a non-oscillating series \\ of elements of $\mathbb{R}$ is convergent\end{tabular}} & \multicolumn{1}{l|}{\begin{tabular}[c]{@{}l@{}}Projection of a non-oscillating series \\ of elements of $\mathbb{R}$ is convergent\end{tabular}} \\ \hline
\end{tabular}}
\end{spacing}
\caption{A summary of the properties of the classes we introduce in this paper. All of them can be used to describe those (natural and tech) phenomena for which Diophantine arithmetics fail.}
\label{tab1}
\end{table}

The paper is divided as follows. Section \ref{sec2} deals with $\{\mathbf{A}_M\}_{M \geq 1}$. Section \ref{sec4} presents the class $\{\mathbf{A}_{-M,M}\}_{M \geq 1}$, while 
$\{\mathbf{B}_{M}\}_{M > 0}$ is discussed in Section \ref{app1}.

\section{Class $\{\mathbf{A}_M\}_{M \geq 1}$.}\label{sec2}
For any real $M\geq 1$, we define the corresponding non-Diophantine preartithmetic as
$$\mathbf{A}_M=(A_M,\oplus,\otimes,\leq_{A_M})$$
having the following properties:
\begin{itemize}
	\item[(i)] The order relation $\leq_{A_M}$ is the restriction to ${A}_M$ of the usual order on the reals;
    \item[(ii)] $A_M \subset \mathbb{R}_+$ has maximal element $M$ and minimal element $0$ with respect to $\leq_{A_M}$, and is such that
    \begin{itemize}
    \item[\textbullet] $0 \in A_M$, which ensures having a multiplicative absorbing and additive neutral element in our set;
    \item[\textbullet] $1\in A_M$, which ensures having a multiplicative neutral element in our set;
    \item[\textbullet] there is at least an element $x\in(0,1)$ such that $x \in A_M$;
    \end{itemize}
    \item[(iii)] It is closed under the following two operations
     $$\oplus:A_M\times A_M\rightarrow A_M, \quad \oplus:(a,b)\mapsto a \oplus b:=\min(M,a+b)$$
     and
    $$\otimes:A_M\times A_M\rightarrow A_M, \quad \otimes:(a,b)\mapsto a \otimes b:=\min(M,a \times b),$$
    where $+$ and $\times$ denote the usual sum and product in $\mathbb{R}$, respectively.
\end{itemize}

\begin{proposition}\label{prop2}
Addition $\oplus$ is associative.
\end{proposition}
\begin{proof}
Pick any $a,b,c \in A_M$. Let $a\oplus b=d$, where $d=a+b$ if $a+b \leq M$ or $d=M$ if $a+b > M$. Let also $b\oplus c=e$, where $e=b+c$ if $b+c \leq M$ or $e=M$ if $b+c > M$. Then,
$$(a\oplus b)\oplus c=d\oplus c=\begin{cases}
d+c=(a+b)+c & \text{if } d+c \leq M\\
M & \text{if } d+c  > M
\end{cases}$$ and $$a\oplus (b\oplus c)=a\oplus e=\begin{cases}
a+e=a+(b+c) & \text{if } a+e \leq M\\
M & \text{if } a+e  > M
\end{cases}.$$
But we know that $(a+b)+c=a+(b+c)$, so the result follows.
\end{proof}

Since addition $\oplus$ is associative we have that, for every $k\in\mathbb{N}$,
$$\bigoplus_{n=1}^kx_n=\min\left(M,\sum_{n=1}^nx_n\right).$$
By imposing on $M$ the relative topology derived from $\mathbb{R}$, we can define
$$\bigoplus^{\infty}_{n=1}x_n:=\lim_{k\rightarrow\infty}\bigoplus_{n=1}^kx_n=\min\left(M,\sum_{n=1}^{\infty}x_n\right).$$

\begin{proposition}\label{prop3}
$A_{M}=[0,M]$.
\end{proposition}
\begin{proof}
Recall that ${A}_M$ is closed under products $\otimes$ and arbitrary summations $\oplus$. Fix any $x \in (0,1)$ and let $z\in(0,M]$. Then, it is always possible to express
$$z=\bigoplus_{n=1}^{\infty}\left(\bigoplus_{j=1}^{k_n}x^n\right),$$
for some $(k_n)\in\mathbb{N}^{\mathbb{N}}$, defined recursively as follows
\begin{itemize}
    \item $k_1=\max\{m\in\mathbb{N}:z>mx\}$;
    \item $k_{j+1}=\max\{m\in\mathbb{N}:z>\sum_{u=1}^jk_ux^u+mx^j\}$.
\end{itemize}
It is clear that the partial sums are always less than $z\leq M$, and so are all nonnegative and well defined. Moreover,
$$\left| z-\bigoplus_{n=1}^{m}\left(\bigoplus_{j=1}^{k_n}x^n\right)\right| \leq x^m\rightarrow0.$$
Finally observe that the result of our two operations is always less or equal than $M$ and, from non-negative numbers, it is impossible to obtain negative numbers.
\end{proof}

This result implies that $\mathbf{A}_M$ is a complete totally ordered semiring. In this Section, we consider the class $\{\mathbf{A}_M\}_{M \geq 1}$ of abstract prearithmetics.

\begin{remark}
Notice that $\mathbf{A}_M$ cannot be a ring because for any $a \in A_M\backslash\{0\}$, it lacks the additive inverse $-a$; this because we defined $A_M$ to be a subset of $\mathbb{R}_+$. Notice also that in this abstract prearithmetic $M$ is an idempotent element, that is, $M \oplus M \oplus \cdots \oplus M=M$.
\end{remark}

\subsection{Overcoming the paradox of the heap.}
The paradox of  the heap is a paradox that arises from vague predicates. A formulation of such paradox (also called the sorites paradox, from the Greek word $\sigma \omega \rho o \varsigma$, ``heap''), given in \cite[Section 2]{Burgin2017}, is the following. 
\begin{itemize}
\item[(1)] One million grains of sand make a heap;
\item[(2)] If one grain of sand is added to this heap, the heap stays the same;
\item[(3)] However, when we add $1$ to any natural number, we always get a new number. 
\end{itemize}
This formulation of the paradox of the heap is proposed by Burgin to inspect whether adding \$$1$ to the assets of a millionaire makes them ``more of a millionaire'', or leaves their fortune unchanged. We use the class $\{\mathbf{A}_M\}_{M \geq 1}$ to address paradox of the heap.  Indeed, it is enough  to take the element of the class for which $M=1000000$, so that when we perform the addition $M \oplus 1$, we get $M$. This conveys the idea that adding a grain of sand to the heap leaves us with a heap. 

The abstract prearithmetic we introduced can also be used to describe phenomena like the one noted by Helmholtz in \cite{helmholtz}: adding one raindrop to another one gives one raindrop, or the one pointed out by Lebesgue (cf. \cite{kline}): putting a lion and a rabbit in a cage, one will not find two animals in the cage later on. In both these cases, it suffices to consider the element  of the class for which $M=1$, so that $1 \oplus 1=1$. 

$\mathbf{A}_M$ allows us also to avoid introducing inconsistent Diophantine arithmetics, that is, arithmetics for which one or more Peano axioms were at the same time true and false. For example, in \cite{rosinger} Rosinger points out that electronic digital computers, when operating on the integers, act according to the usual Peano axioms for $\mathbb{N}$ plus an extra ad-hoc axiom, called the machine infinity axiom. The machine infinity axiom states that there exists $\breve{M} \in \mathbb{N}$ far greater than $1$ such that $\breve{M}+1=\breve{M}$ (for example, $\breve{M}=2^{31}-1$ is the maximum positive value for a $32$-bit signed binary integer in computing). Clearly, Peano axioms and the machine infinity axiom together give rise to an inconsistency, which can be easily avoided by working with $\mathbf{A}_{\breve{M}}$.

In \cite{vanbendegem}, Van Bendegem developed an inconsistent axiomatic arithmetic similar to the ``machine'' one described in \cite{rosinger}. He changed the Peano axioms so that a number that is the successor of itself exists. In particular, the fifth Peano axiom states that if $x=y$, then $x$ and $y$ are the same number. In the system of Van Bendegem, starting from some number $n$, all its successors will be equal to $n$. Then, the statement $n=n+1$ is considered as both true and false at the same time, giving rise to an inconsistency. It is immediate to see how this inconsistency can be overcome by working with any abstract prearithmetic $\mathbf{A}_M$ in our class.

\subsection{$\mathbf{A}_M$ is weakly projective with respect to $\mathbf{R_+}$.}
Pick any $\mathbf{A}_{M^\prime} \in \{\mathbf{A}_M\}$, and consider $\mathbf{R_+}=(\mathbb{R}_+,+,\times,\leq_{\mathbb{\mathbb{R}_+}})$. Consider then the functions:
\begin{itemize}
\item $g:A_M \rightarrow \mathbb{R}_+$, $a \mapsto g(a) \equiv a$, so $g$ is the identity function $\left.\text{Id}\right|_{A_M}$;
\item $h:\mathbb{R}_+ \rightarrow A_M$, $a \mapsto h(a):=\min(M,a)$.
\end{itemize}
Now, if we compute $h(g(a)+g(b))$, for all $a,b \in A_M$, we have that $$h(g(a)+g(b))=h(a+b)=\min(M,a+b)=a \oplus b.$$ Similarly, we show that $h(g(a)\times g(b))=a \otimes b$. Hence, addition and multiplication in $\mathbb{R}_+$ are weakly projected onto addition and multiplication in $A_{M^\prime}$, respectively. So, we can conclude that $\mathbf{A}_{M^\prime}$ is weakly projective with respect to $\mathbf{R_+}$, for all ${M^\prime} \geq 1$. 

\subsection{Series of elements of $\mathbf{A}_M$.} 
Consider any series $\bigoplus_n a_n$ of elements of $A_M$. Its corresponding series $\sum_n g(a_n)$ in $\mathbb{R}_+$ can only be convergent or divergent to $+\infty$. It cannot be divergent to $-\infty$ because we are summing positive elements only, and it cannot be neither convergent nor divergent (i.e. it cannot oscillate), because the elements of the series cannot alternate their sign. But since $A_M$ has a maximal and a minimal element, $M$ and $0$ respectively, this means that $\bigoplus_n a_n$ is always convergent.

\subsection{Weak projection of series in $\mathbf{R_+}$ onto $\mathbf{A}_M$.}
In this Section, we show that the weak projection of any series of elements of $\mathbb{R}_+$ converges in $A_M$, for all $M \geq 1$. This is an exciting result because it allows the scholar that needs a particular series to converge in their analysis to reach that result by performing a weak projection of the series onto $\mathbf{A}_M$, and then continue the analysis in $\mathbf{A}_M$.

Consider any series $\sum_n a_n$ of elements of $\mathbb{R}_+$. Similarly to what we pointed out before, it can be convergent or divergent to $+\infty$. It cannot be divergent to $-\infty$ because we are summing positive elements only, and it cannot be neither convergent nor divergent (i.e. it cannot oscillate), because the elements of the series cannot alternate their sign. Let us then show that the weak projection of $\sum_{n=1}^\infty a_n:=\lim\limits_{k \rightarrow \infty} \sum_{n=1}^k a_n$ is convergent.

First, suppose that $\sum_{n=1}^\infty a_n =L < M$. Then, $h( \sum_{n=1}^\infty   a_n  ) =h(L)=L$.

Then, let $\sum_{n=1}^\infty a_n =L  \geq M$. We have that $h( \sum_{n=1}^\infty  a_n  ) =h(L)=M$, where in both cases the last equality comes from the definition of $h$. 

Finally, suppose that $\sum_{n=1}^\infty  a_n =\infty$. Then, $$h\left( \sum_{n=1}^\infty  a_n \right) = h\left( \lim_{k \rightarrow \infty} \sum_{n=1}^k  a_n \right)= \lim_{k \rightarrow \infty} h\left( \sum_{n=1}^k  a_n \right) = M,$$where the second equality comes from $h$ being continuous, and the last equality comes from the fact that function $h$ is constant once its argument is equal to $M$. To check that $h$ is continuous, we just need to check that it is continuous at $M$; but this is immediate, since $\lim_{a \rightarrow M^-} h(a)=M=\lim_{a \rightarrow M^+} h(a)$. 

The following lemma comes immediately from Proposition \ref{prop2}.


\begin{lemma}\label{lem1}
For any series $\sum_{n=1}^\infty a_n$ of elements of $\mathbb{R}_+$,  for any $k \in \mathbb{N}$,
$$h \left( \sum_{n=1}^\infty a_n \right) = h \left( \sum_{n=1}^k a_n \right) \oplus h \left( \sum_{n=k+1}^\infty a_n \right).$$
\end{lemma}

\section{Class $\{\mathbf{A}_{-M,M}\}_{M \geq 1}$}\label{sec4}
For any real $M \geq 1$, we define the corresponding non-Diophantine prearithmetic as
$$\mathbf{A}_{-M,M}=(A_{-M,M},\boxplus,\boxtimes,\leq_{A_{-M,M}})$$
having the following properties:
\begin{itemize}
\item[(i)] The order relation $\leq_{A_{-M,M}}$ is the restriction to ${A}_{-M,M}$ of the usual order on the reals;
\item[(ii)] $A_{-M,M} \subset \mathbb{R}$ has maximal element $M$ and minimal element $-M$ with respect to $\leq_{A_{-M,M}}$, and is such that
\begin{itemize}
\item[\textbullet] $0 \in A_{-M,M}$, which ensures having a multiplicative absorbing and additive neutral element in our set;
\item[\textbullet] $1 \in A_{-M,M}$, which ensures having a multiplicative neutral element in our set;
\item[\textbullet] $-M < 0 < M$ and there is at least an element $x \in (0,1)$ such that $x \in A_{-M,M}$;
\item[\textbullet] $b \in A_{-M,M} \implies -b \in A_{-M,M}$;
\end{itemize}
\item[(iii)] It is closed under the following two operations
\begin{align*}
\boxplus: A_{-M,M} \times A_{-M,M} &\rightarrow A_{-M,M},\\ (a,b) &\mapsto a \boxplus b:=\begin{cases}
a+b & \text{if } a+b \in  [-M,M]\\
M & \text{if } a+b > M\\
-M & \text{if } a+b < -M
\end{cases},
\end{align*}
where $+$ denotes the usual sum in $\mathbb{R}$, and
\begin{align*}
\boxtimes: A_{-M,M} \times A_{-M,M} &\rightarrow A_{-M,M}, \\ (a,b) &\mapsto a \boxtimes b:=\begin{cases}
a \times b & \text{if } a \times b \in  [-M,M]\\
M & \text{if } a \times b > M\\
-M & \text{if } a \times b <-M
\end{cases},
\end{align*}
where $\times$ denotes the usual product in $\mathbb{R}$.
\end{itemize}
Since $\mathbf{A}_{-M,M}$ is closed under products $\boxtimes$ and arbitrary summations $\boxplus$, it follows, along the lines of Proposition \ref{prop3}, that $A_{-M,M}=[-M,M]$, for all $M \geq 1$. Notice also that from the definition of addition $\boxplus$ it follows immediately that $\boxplus$ is commutative, but $\mathbf{A}_{-M,M}=(A_{-M,M},\boxplus,\boxtimes,\leq_{A_{-M,M}})$ is not a semiring. Indeed, addition $\boxplus$ is not associative. An easy counterexample is the following: $-2 \boxplus (M \boxplus 1)=-2 \boxplus M=M-2$, while $(-2 \boxplus M) \boxplus 1=-1 \boxplus M=M-1$.

Despite that, the elements of $\{\mathbf{A}_{-M,M}\}_{M \geq 1}$ can still be useful. Any $\mathbf{A}_{-{M^\prime},{M^\prime}}$ still solves the paradox of the heap. If we then consider the functions
\begin{itemize}
\item $s:A_{-M,M} \rightarrow \mathbb{R}$, $a \mapsto s(a) \equiv a$, so $s$ is the identity function $\left.\text{Id}\right|_{A_{-M,M}}$;
\item $t:\mathbb{R} \rightarrow A_{-M,M}$, $a \mapsto t(a):=\begin{cases}
a &\text{if } a \in [-M,M]\\
M &\text{if } a > M\\
-M &\text{if } a < -M
\end{cases}$,
\end{itemize}
it is immediate to see that, for all $M^\prime \geq 1$, $\mathbf{A}_{-{M^\prime},{M^\prime}}$ is weakly projective with respect to the real arithmetic $\mathbf{R}=(\mathbb{R},+,\times,\leq_{\mathbb{R}})$, with projector $s$ and coprojector $t$. 
\begin{proposition}\label{prop1}
Pick any real $M \ge 1$, and any sequence $(x_n) \in \mathbb{R}^\mathbb{N}$. For every $n$, define $y_n:=\sum_{n \leq k} x_n$. Then, the weak projection of $(y_n)$ converges in $\mathbf{A}_{-M,M}$ if and only if  one of the following holds:
\begin{itemize}
\item[(i)] $\liminf_n y_n \geq M$;
\item[(ii)] $\lim y_n$ exists and belongs to $(-M,M)$;
\item[(iii)] $\limsup_n y_n \leq -M$.
\end{itemize}
\end{proposition}
\begin{proof}

The fact that $t$ is continuous is obvious. Suppose then that $\lim_n t(y_n)=t(\lim_n y_n) \in (-M,M)$, where the equality comes from the continuity of $t$. By definition of $t$, we have that either $\lim_n y_n$ belongs to $(-M,M)$, verifying condition (ii), or it does not. In this latter case, it may be that either $\lim_n y_n \geq M$, verifying condition (i), or $\lim_n y_n \leq -M$, verifying condition (iii).

The fact that condition (ii) implies  $\lim_n t(y_n) \in (-M,M)$ is immediate by the definition of function $t$. Suppose now that $\liminf_n y_n=\infty > M$. This implies that $\lim_n y_n=\infty$. But then, $\lim_n t(y_n)=t(\lim_n y_n)=M$, where the first equality comes form the continuity of $t$, and the last equality from the definition of $t$. A similar argument shows that if $\limsup_n y_n =-\infty$, then $\lim_n t(y_n)=-M$. If $\liminf_n y_n= L \in [M,\infty)$, the previous argument shows that $\lim_n t(y_n) =M$, while if $\limsup_n y_n= -L \in (-\infty,-M]$, the previous argument shows that $\lim_n t(y_n) =-M$. This concludes the proof.
\end{proof}
The following corollary comes immediately from Proposition \ref{prop1}.
\begin{corollary}
Pick any real $M \ge 1$. The weak projection of any series of elements of $\mathbb{R}$ is absolutely convergent in $\mathbf{A}_{-M,M}$. In addition, the weak projection of any series of elements of $\mathbb{R}$ that is either convergent or divergent (to $+\infty$ or $-\infty$), converges in $\mathbf{A}_{-M,M}$.
\end{corollary}

Notice that, because addition $\boxplus$ is not associative, Lemma \ref{lem1} does not hold in this class of abstract prearithmetics. This means that there may exist a series $\sum_{n=1}^\infty a_n$ of elements of $\mathbb{R}$ such that 
$$t \left( \sum_{n=1}^\infty a_n \right) \neq t \left( \sum_{n=1}^k a_n \right) \boxplus t \left( \sum_{n=k+1}^\infty a_n \right).$$
Hence, we need to project the entire series in the exact order we want the elements to be summed, otherwise Proposition \ref{prop1} may not hold.

There is a tradeoff in using this class instead of $\{\mathbf{A}_M\}$. We still manage to resolve the paradox of the heap, and we further show that the weak projection of series that diverge also to $-\infty$, converges in $\mathbf{A}_{-M,M}$, for all $M \geq 1$. The shortcoming, however, is that we lose associativity, so working with elements of $A_{-M,M}$ may be difficult. Ultimately, the choice of one or the other will depend on the application the scholar has in mind.

\section{Class $\{\mathbf{B}_M\}_{M > 0}$.}\label{app1}
In this Section we present a class of abstract prearithmetics $\{\mathbf{B}_M\}_{M > 0}$ where every element is a complete totally ordered semiring, and such that the projection of a convergent or divergent series (to $+\infty$ or $-\infty$) of elements of $\mathbb{R}$ converges. Its elements can be used to solve the paradox of the heap. 


For every real $M > 0$, we define the corresponding non-Diophantine prearithmetic as 
$$\mathbf{B}_M=(B_M,\dotplus,\dottimes,\leq_{B_M})$$ 
having the following properties:
\begin{itemize}
\item[(i)] The order relation $\leq_{B_M}$ is the restriction to $B_M$ of the usual order on the reals;
\item[(ii)] $B_M=[0,M]$;
\item[(iii)] Let $\overline{\mathbb{R}}:=[-\infty,\infty]$ and consider the function
$$f:\overline{\mathbb{R}} \rightarrow B_M, \quad x \mapsto f(x):=\begin{cases}
M \left(\frac{\arctan(x)}{\pi}+\frac{1}{2}\right) & \text{if } x \in \mathbb{R}\\
M & \text{if } x=\infty\\
0 & \text{if } x=-\infty
\end{cases}$$
and its inverse
$$f^{-1}:B_M \rightarrow \overline{\mathbb{R}}, \quad x \mapsto f^{-1}(x):=\begin{cases}
\tan\left( \frac{\pi}{M} \left( x-\frac{M}{2} \right) \right) & \text{if } x \in (0,M)\\
\infty & \text{if } x=M\\
-\infty & \text{if } x=0
\end{cases}.$$
Then, $\mathbf{B}_M$ is closed under the following two operations
$$\dotplus:B_M \times B_M \rightarrow B_M, \quad (a,b) \mapsto a \dotplus b:= f\left( f^{-1}(a)+f^{-1}(b) \right),$$
where $+$ denotes the sum in $\overline{\mathbb{R}}$, and
$$\dottimes:B_M \times B_M \rightarrow B_M, \quad (a,b) \mapsto a \dottimes b:= f\left( f^{-1}(a)\times f^{-1}(b) \right),$$
where $\times$ denotes the product in $\overline{\mathbb{R}}$.
\end{itemize}

Notice that we do not ``force'' $0$ and $M$ to be the boundary elements of $B_M$; they come naturally from the way addition $\dotplus$ and multiplication $\dottimes$ are defined. In addition, we have that by construction 
$\mathbf{B}_M$ is projective with respect to $\overline{\mathbf{R}}=(\overline{\mathbb{R}},+,\times,\leq_{\overline{\mathbb{R}}})$, and that its generator induces an homeomorphism between $\overline{\mathbb{R}}$ and $[0,M]$. This tells us immediately that  $(B_M,\dotplus,\dottimes,\leq_{B_M})$ is a complete totally ordered semiring, so addition $\dotplus$ and multiplication $\dottimes$ are associative. The fact that $+\infty$ and $-\infty$ in $\overline{\mathbb{R}}$ correspond to $M$ and $0$, respectively, 
tells us that the projection $f(\sum_n a_n)$ of any series $\sum_n a_n$ of elements of $\mathbb{R}$ converges in $\mathbf{B}_M$, as long as $\sum_n a_n$ does not oscillate. The elements of $\mathbf{B}_M$ can be used to solve the paradox of the heap; to see this, notice that $M \dotplus a=M$, for all $a \in B_M$ and all $M>0$. Also, $M$ is an idempotent element of $B_M$, for all $M>0$: $M \dotplus M \dotplus \cdots \dotplus M=M$.





\bibliographystyle{plain}
\bibliography{non_dioph} 

\end{document}